\documentclass[10pt,a4paper]{amsart}
\usepackage{palatcm}
\usepackage{amsthm,amssymb,amsmath}

\allowdisplaybreaks

\renewcommand{\phi}{\varphi}
\newcommand{\A}{\mathcal{A}}

\newcommand{\K}{\mathcal{K}}
\newcommand{\PP}{\mathcal{P}}
\newcommand{\PPb}{\overline{\mathcal{P}}}
\newcommand{\tU}{\widetilde{U}}
\newcommand{\Mat}{\mathrm{Mat}}
\renewcommand{\H}{\mathcal{H}}

\newcommand{\R}{\mathbb{R}}
\newcommand{\C}{\mathbb{C}}

\newcommand{\sgn}{\mathrm{sign}}
\newcommand{\Symb}{\mathrm{Symb}}
\newcommand{\Psit}{\widetilde\Psi}
\newcommand{\LLlim}{\operatorname{L^2-lim}}
\newcommand{\Llim}{\operatorname{L^1-lim}}

\newcommand{\Hom}{\mathrm{Hom}}
\newcommand{\vol}{\mathrm{vol}}

\newcommand{\supp}{\mathrm{supp}}

\newcommand{\id}{\mathrm{id}}
\renewcommand{\epsilon}{\varepsilon}
\newcommand{\eps}{\varepsilon}
\newcommand{\<}{\left\langle}
\renewcommand{\>}{\right\rangle}

\theoremstyle{plain}
\newtheorem{thm}{Theorem}[section]

\newtheorem{lem}[thm]{Lemma}

\theoremstyle{definition}
\newtheorem{dfn}[thm]{Definition}
\newtheorem{rem}[thm]{Remark}

\begin{document}

\title{{Conformal Structures in Noncommutative Geometry}}

\author{Christian B\"ar}

\address{Universit\"at Potsdam\\
Institut f\"ur Mathematik\\
Am Neuen Palais 10\\
14469 Potsdam\\
Germany}

\email{baer@math.uni-potsdam.de}


\date{\today}

\begin{abstract}
It is well-known that a compact Riemannian spin manifold $(M,g)$ can be
reconstructed from its canonical spectral triple $(C^\infty(M),L^2(M,\Sigma
M),D)$ where $\Sigma M$ denotes the spinor bundle and $D$ the Dirac operator.
We show that $g$ can be reconstructed up to conformal equivalence from
$(C^\infty(M),L^2(M,\Sigma M),\sgn(D))$.
\end{abstract}

\maketitle

\section{Introduction}

To a closed Riemannian manifold $(M,g)$ with a fixed spin structure one can
associate its canonical spectral triple $(\A,\H,D)$ where $\A=C^\infty(M)$ is
the pre-$C^*$-algebra of smooth complex functions on $M$, $\H=L^2(M,\Sigma M)$
is the Hilbert space of square integrable complex spinors on $M$ and $D$ is
the Dirac operator.
It is not hard to see that $(M,g)$ can be reconstructed from $(\A,\H,D)$.
While $\A$ determines $M$ as a differentiable manifold, the pair $(\H,D)$
together with the action of $\A$ on $\H$ determine the geometry.
This is the starting point of noncommutative geometry because the concept of
spectral triples is not confined to commutative algebras $\A$.

The action of $\A$ on $\H$ (and the commutators with $D$) are important in the
reconstruction of the metric.
The pair $(\H,D)$ alone is up to unitary equivalence given by the eigenvalues
of the Dirac operator $D$.
These eigenvalues contain quite a bit of geometric information but they
are not sufficient to determine the metric, not even the topology of $M$.
For example, there are explicit examples of Dirac isospectral manifolds with
nonisomorphic fundamental groups \cite{B}.

Therefore it is debatable whether attempts to replace the metric by the Dirac
eigenvalues as basic dynamical variables in General Relativity can succeed
\cite{LR1,LR2}.
In the present note we will show that one can reconstruct the metric up to
conformal equivalence using a modification of canonical spectral triples
where the Dirac operator $D$ is replaced by $\sgn(D)$.
This is the difference of two spectral projections of $D$; 
one need not know a single Dirac eigenvalue.
This indicates that most information of a spectral triple is contained in the
interplay of $D$ with the action of $\A$, not in the spectrum of $D$.

Our approach is not to be confused with Connes' construction of a natural
Fredholm module on even-dimensional conformal manifolds \cite[Ch.~IV.4]{C}.
His construction is based on the fact that the Hodge $*$-operator on
differential forms of middle degree is conformally invariant.
We use the spinorial Dirac operator just as for spectral triples which works
in odd dimensions as well as in even dimensions.
We hope that our approach will open a door to noncommutative conformal
geometry just as spectral triples are the basic objects of noncommutative
Riemannian geometry.

The paper is organized as follows.
In the next section we introduce some terminology and recall the
reconstruction of the metric from the spectral triple.
In the third section we state and prove the main theorem.
It says that two Riemannian metrics $g$ and $g'$ on $M$ are conformally
equivalent if and only if their Fredholm modules $(\H,\sgn(D))$ and
$(\H',\sgn(D'))$ are weakly unitarily equivalent.
It seems to be a folklore fact among noncommutative geometers that $\sgn(D)$
determines the conformal structure.
Nevertheless, we found it worthwhile to give a precise formulation and proof
of this fact.

Our proof is based on symbolic calculus of the pseudo-differential operator
$\sgn(D)$.
The main technical difficulty is caused by the necessity to use a standard
fact for principal symbols in the case of low regularity
(Lemma~\ref{lem:irreg}).
For the sake of better readability these more technical considerations are
postponed to the last section.

{\bf Acknowledgements.}
This note arose from discussions at the Mini-Workshop on ``Dirac Operators in
Differential and Noncommutative Geometry'' in 2006 in Oberwolfach.
It is a pleasure to thank the participants, in particular Alexander Strohmaier,
for important remarks.
Many thanks go to Elmar Schrohe and Bert-Wolfgang Schulze for helpful remarks
on pseudo-differential operators.
I would also like to thank the Mathematische Forschungsinstitut Oberwolfach
for its hospitality and SPP 1154 funded by Deutsche Forschungsgemeinschaft
for financial support.

\section{Operator modules}

\begin{dfn}
Let $\A$ be a pre-$C^*$-algebra.
An \emph{operator module} of $\A$ is a pair $(\H,D)$ where $\H$ is a
complex Hilbert space together with a $*$-representation of $\A$ on $\H$ by
bounded linear operators and $D$ is a (possibly unbounded but densely defined)
linear operator in $\H$ whose domain is left invariant by the action of $\A$.
\end{dfn}

For the $*$-representation of $\A$ on $\H$ we will write $(a,h) \mapsto a\cdot
h = ah$.
This concept of operator modules contains Connes' spectral triples
\cite[Def.~9.16]{GBFV} as well as Fredholm modules which are used to define
K-homology \cite[Ch.~8]{HR}. 

\begin{dfn}
Two operator modules $(\H,D)$ and $(\H',D')$ of $\A$ are called \emph{unitarily
equivalent} if there exists a unitary isomorphism $U:\H \to \H'$ such that for all $a\in\A$ and $h\in\H$
\begin{itemize}
\item[(i)]
$U(ah) = aU(h)$
\item[(ii)]
$D' = UDU^{-1}$
\end{itemize}
In condition (ii) it is understood that $U$ maps the domain of definition of
$D$ onto the domain of $D'$.
If we replace (ii) by the weaker requirement that $D'-UDU^{-1}$ be a compact
operator, then we call $(\H,D)$ and $(\H',D')$ \emph{weakly unitarily
equivalent}.
\end{dfn}

Now let $(M,g)$ be a compact Riemannian spin manifold.
By a spin manifold we mean a manifold with a fixed spin structure.

We put $\A := C^\infty(M)$, the algebra of smooth $\C$-valued functions, $\H
:= L^2(M,\Sigma M)$, the Hilbert space of square integrable complex spinor
fields, and let $D$ be the Dirac operator. 
Then $(\H,D)$ is an operator module of $\A$ and $(\A,\H,D)$ is called the
\emph{canonical spectral triple} of $(M,g)$. 

It is well-known that $(\H,D)$ determines the Riemannian metric on $M$.
The argument for this is as follows, compare \cite[p.~544]{C}:
The metric distance $d(x,y)$ of two points in $(M,g)$ is given by
$$
d(x,y) = \sup\{ |a(x)-a(y)|\,|\, a\in\A \mbox{ with } \|[D,a]\| \leq 1\} .
$$
Notice that $d(x,y)=\infty$ if and only if $x$ and $y$ lie in different
connected components of $M$.
Now if $(\H,D)$ and $(\H',D')$ are unitarily equivalent, then
$\|[D,a]\|=\|[D',a]\|$ for any $a$ and hence $d=d'$.
Since the metric distance function $d$ determines the Riemannian metric
we have $g=g'$.
Conversely, if $g=g'$, then $(\H,D)=(\H',D')$.
This discussion shows that the following are equivalent:
\begin{itemize}
\item $g=g'$,
\item $(\H,D)$ and $(\H',D')$ are unitarily equivalent,
\item $(\H,D)=(\H',D')$.
\end{itemize}

The aim of this note is to find a similar criterion for $g$ and $g'$
being conformal to each other rather than equal.

\section{Conformal structures}

Two Riemannian metrics $g$ and $g'$ on $M$ are called \emph{conformally
equivalent}, if there exists a smooth function $v:M \to \R$ such that $g' =
e^{2v}g$.

Let $\sgn : \R \to \R$ be the \emph{sign function}, 
$$
\sgn(t) = \left\{
          \begin{array}{cc}
                 1, & \mbox{ if } t>0, \cr
                 0, & \mbox{ if } t=0, \cr
                 -1, & \mbox{ if } t<0.
          \end{array}\right.
$$
To a compact Riemannian spin manifold we associate the operator module
$(\H,\sgn(D))$ of $\A = C^\infty(M)$ where again $\H = L^2(M,\Sigma M)$ is the
Hilbert space of square integrable complex spinor fields, and $D$ is the Dirac
operator.
We call $(\H,\sgn(D))$ the \emph{canonical Fredholm module} of $(M,g)$.

\begin{thm}\label{thm:main}
Let $M$ be a compact spin manifold.
Let $g$ and $g'$ be Riemannian metrics on $M$ and let $(\H,\sgn(D))$ and
$(\H',\sgn(D'))$ be the corresponding canonical Fredholm modules of
$\A=C^\infty(M)$. 

Then $g$ and $g'$ are conformally equivalent if and only if $(\H,\sgn(D))$ and
$(\H',\sgn(D'))$ are weakly unitarily equivalent.
\end{thm}

\begin{proof}
Let $v$ be a smooth function on $M$ such that $g' = e^{2v}g$. 
There exists a fiberwise isometric vector bundle isomorphism $\Psi : \Sigma M
\to \Sigma'M$ between the spinor bundles of $M$ with respect to the metrics $g$
and $g'$ such that we have for the Dirac operators
\begin{equation}\label{eq:Diracconf1}
D' = 
e^{-\frac{n+1}{2}v} \circ\Psi \circ D \circ \Psi^{-1} \circ e^{\frac{n-1}{2}v},
\end{equation}
see e.~g.\ \cite[Sec.~4]{H} for details.
Here $n$ denotes the dimension of $M$.
The fiberwise vector bundle isometry $\Psi$ does not induce a Hilbert space
isometry $L^2(M,\Sigma M) \to L^2(M,\Sigma'M)$ because the volume forms of $g$
and $g'$ are different, $\vol' = e^{nv} \vol$.
To correct this we put $\Psit := e^{-\frac{n}{2}v}\cdot\Psi$.
Now (\ref{eq:Diracconf1}) translates into
\begin{equation}\label{eq:Diracconf2}
D' = 
e^{-\frac{1}{2}v} \circ\Psit \circ D \circ \Psit^{-1} \circ e^{-\frac{1}{2}v}.
\end{equation}
Denote the unitary isomorphism induced by $\Psit$ by $U: L^2(M,\Sigma M) \to
L^2(M, \Sigma'M)$. 
Equation (\ref{eq:Diracconf2}) implies for the principal symbols
\begin{eqnarray*}
\sigma_{D'}(\xi) 
&=&
e^{-\frac{1}{2}v(x)} \cdot\Psit(x) \cdot \sigma_D(\xi) 
\cdot \Psit^{-1}(x) \cdot e^{-\frac{1}{2}v(x)}\\
&=&
e^{-v(x)} \cdot\Psit(x) \cdot \sigma_D(\xi) 
\cdot \Psit^{-1}(x)
\end{eqnarray*}
for $\xi\in T^*_xM$.
Now $\sgn(D)$ and $\sgn(D')$ are classical pseudo-differential operators of
order $0$, compare the argument in \cite[p.~48]{APS} based on the work in
\cite{S2}. 
See e.~g.\ \cite{RS} for an introduction to these operators.
For the principal symbol we have for nonzero $\xi \in T^*_xM$ 
$$
\sigma_{\sgn(D)}(\xi) = \frac{\sigma_D(\xi)}{|\sigma_D(\xi)|}.
$$
Here $|\sigma_D(\xi)|$ denotes the operator
$\sqrt{\sigma_D(\xi)^2}$ and
$\frac{\sigma_D(\xi)}{|\sigma_D(\xi)|}$ means $\sigma_D(\xi)\circ
|\sigma_D(\xi)|^{-1}=|\sigma_D(\xi)|^{-1}\circ\sigma_D(\xi)$.
Notice that $\sigma_D(\xi) \in \Hom(\Sigma_xM,\Sigma_xM)$ is a symmetric
operator. 
Moreover,
\begin{eqnarray*}
\sigma_{\sgn(D')}(\xi) 
&=& 
\frac{e^{-v(x)} \cdot\Psit(x) \cdot \sigma_D(\xi)\cdot \Psit^{-1}(x)}
{|e^{-v(x)} \cdot\Psit(x) \cdot \sigma_D(\xi)\cdot \Psit^{-1}(x)|}\\
&=&
\frac{\Psit(x) \cdot \sigma_D(\xi)\cdot \Psit^{-1}(x)}
{|\Psit(x) \cdot \sigma_D(\xi)\cdot \Psit^{-1}(x)|}\\
&=&
\frac{\Psit(x) \cdot \sigma_D(\xi)\cdot \Psit^{-1}(x)}
{\Psit(x) \cdot|\sigma_D(\xi)|\cdot \Psit^{-1}(x)} \\
&=&
\Psit(x) \cdot\frac{ \sigma_D(\xi)}
{|\sigma_D(\xi)|}\cdot \Psit^{-1}(x)\\
&=&
\sigma_{U\circ\sgn(D)\circ U^{-1}}(\xi).
\end{eqnarray*}
Hence $\sgn(D')-U\circ\sgn(D)\circ U^{-1}$ is a classical pseudo-differential
operator of order $0$ with vanishing principal symbol
and thus compact as an operator on $L^2(M,\Sigma'M)$, compare the exact
sequence (\ref{eq:ExSeq}) below.
This shows that the canonical Fredholm modules $(L^2(M,\Sigma M),\sgn(D))$ and
$(L^2(M,\Sigma'M),\sgn(D'))$ of $C^\infty(M)$ are weakly unitarily equivalent.

Conversely, assume that the canonical Fredholm modules
$(L^2(M,\Sigma M),\sgn(D))$ and
$(L^2(M,\Sigma'M),\sgn(D'))$ of $C^\infty(M)$ are weakly unitarily equivalent.
The unitary isomorphism $U$ commutes with the action of $C^\infty(M)$, hence
$U$ is induced by a section $\Psi \in L^\infty(M,\Hom(\Sigma M,\Sigma'M))$,
see Lemma~\ref{lem:Kommutante}.
Since $U$ is invertible, $\Psi$ is invertible almost everywhere and $\Psi^{-1}
\in L^\infty(M,\Hom(\Sigma'M,\Sigma M))$. 

The principal symbol of a Dirac operator is given by Clifford multiplication
with respect to the metric $g$, more precisely, 
$\sigma_D(\xi) = ic_g(\xi)$ for all $\xi\in T^*M$.
Clifford multiplication has the property $c_g(\xi)c_g(\eta) +
c_g(\eta)c_g(\xi) + 2 g(\xi,\eta) = 0$ for all $\xi,\eta \in T^*_xM$ and all
$x\in M$.
Hence the principal symbol of $\sgn(D)$ is given by $\sigma_{\sgn(D)}(\xi) =
\frac{ic_g(\xi)}{\|\xi\|_g}$ for $\xi\in T^*_xM$, $\xi\not=0$.
By Lemma~\ref{lem:irreg} 
$$
\sigma_{U\circ\sgn(D)\circ U^{-1}}(\xi) =
\Psi(x) \cdot \sigma_{\sgn(D)}(\xi) \cdot \Psi^{-1}(x) .
$$
Since $\sgn(D')$ and $U\circ\sgn(D)\circ U^{-1}$ differ by a compact operator
we have $\sigma_{\sgn(D')} = \sigma_{U\circ\sgn(D)\circ U^{-1}}$.
This means
$$
\frac{c_{g'}(\xi)}{\|\xi\|_{g'}}
=
\Psi(x)\cdot \frac{c_g(\xi)}{\|\xi\|_g} \cdot \Psi^{-1}(x)
$$
for all nonzero $\xi$.
Therefore
\begin{eqnarray*}
\frac{-2g'(\xi,\eta)}{\|\xi\|_{g'}\|\eta\|_{g'}}
&=&
\frac{c_{g'}(\xi)}{\|\xi\|_{g'}}\frac{c_{g'}(\eta)}{\|\eta\|_{g'}}
+ \frac{c_{g'}(\eta)}{\|\eta\|_{g'}}\frac{c_{g'}(\xi)}{\|\xi\|_{g'}}\\
&=&
\Psi(x) \cdot \left(
\frac{c_{g}(\xi)}{\|\xi\|_{g}}\frac{c_{g}(\eta)}{\|\eta\|_{g}}
+ \frac{c_{g}(\eta)}{\|\eta\|_{g}}\frac{c_{g}(\xi)}{\|\xi\|_{g}}
\right)\cdot \Psi^{-1}(x)\\
&=&
\Psi(x) \cdot \left(
\frac{-2g(\xi,\eta)}{\|\xi\|_{g}\|\eta\|_{g}}
\right)\cdot \Psi^{-1}(x)\\
&=&
\frac{-2g(\xi,\eta)}{\|\xi\|_{g}\|\eta\|_{g}} .
\end{eqnarray*}
The last equation holds because
$\frac{-2g(\xi,\eta)}{\|\xi\|_{g}\|\eta\|_{g}}$ is scalar and thus commutes
with $\Psi(x)$.
This proves that $g$ and $g'$ are conformally equivalent.
\end{proof}

\begin{rem}
Let $\pi_+ : \R \to \R$ and $\pi_0:\R\to\R$ be given by 
$$
\pi_+(t) = \left\{
\begin{array}{cl}
1, & \mbox{ if } t>0,\\
0, & \mbox{ if } t\leq 0.
\end{array}
\right.
\quad\mbox{ and }\quad\quad
\pi_0(t) = \left\{
\begin{array}{cl}
1, & \mbox{ if } t=0,\\
0, & \mbox{ if } t\not= 0.
\end{array}
\right. .
$$
Since $\pi_+(D) = \frac12(\sgn(D)+\id-\pi_0(D))$ and since $\pi_0(D)$ is a
finite rank operator and hence compact we see that $(\H,\sgn(D))$ and
$(\H',\sgn(D'))$ are weakly unitarily equivalent if and only if $(\H,\pi_+(D))$
and $(\H',\pi_+(D'))$ are weakly unitarily equivalent.
Therefore one can replace $\sgn(D)$ by the spectral projector $\pi_+(D)$ in
Theorem~\ref{thm:main}. 
\end{rem}

\section{Three auxiliary lemmas}

Throughout this section let $E\to M$ and $E'\to M$ be
Hermitian vector bundles over the closed Riemannian manifold $M$.

\begin{lem}\label{lem:Kommutante}
Let $U :L^2(M,E) \to L^2(M,E')$ be a bounded linear map.
Suppose $U(a\phi)=aU(\phi)$ for all $a\in C^\infty(M)$ and all $\phi\in
L^2(M,E)$.

Then there exists a unique $\Psi \in L^\infty(M,\Hom(E,E'))$ such that 
$$
(U\phi)(x) = \Psi(x)\phi(x)
$$
for almost all $x\in M$ and for all $\phi\in L^2(M,E)$.
\end{lem}

\begin{proof}
Uniqueness of $\Psi$ is obvious.
To show existence we choose trivializing complements $F$ and $F'$ for the
bundles $E$ and $E'$, i.~e., $F$ and $F'$ are Hermitian vector bundles over
$M$ such that $E\oplus F$ and $E'\oplus F'$ are trivial bundles.
Without loss of generality we assume that $E\oplus F$ and $E'\oplus F'$ have
equal rank $N$.
We extend $U$ trivially to an operator 
$$
\tU = \begin{pmatrix}U & 0 \cr 0 & 0\end{pmatrix}
: L^2(M,E\oplus F)= L^2(M,\C^N) \to L^2(M,E'\oplus F')= L^2(M,\C^N) .
$$
The extended operator $\tU$ still commutes with the action of $C^\infty(M)$ on
$L^2(M,\C^N)$ and hence also with the action of $L^\infty(M)$, the von Neumann
algebra generated by $C^\infty(M)$.
It is well-known that the commutant of $L^\infty(M)$ on the Hilbert space
$L^2(M,\C^N)$ is given by $L^\infty(M,\Mat(\C,N))$, see e.~g.\ \cite[p.~61]{F}.
This proves the lemma.
\end{proof}

We denote the space of classical pseudo-differential operators of order $0$
acting on sections in $E$ by $\PP(M,E)$.
Each element of $\PP(M,E)$ extends to a bounded linear map on the Hilbert
space $L^2(M,E)$.
We denote the closure of $\PP(M,E)$ with respect to the $L^2$-operator
norm by $\PPb(M,E)$.
There is a well-known exact sequence \cite[Thm.~11.1]{S}
\begin{equation}\label{eq:ExSeq}
0 \longrightarrow \K(L^2(M,E)) \longrightarrow \PPb(M,E)
\stackrel{\sigma_\bullet}{\longrightarrow} \Symb(M,E) \longrightarrow 0
\end{equation}
where $\K$ stands for compact operators and $\Symb(M,E) = \{\sigma\in
C^0(S^*M,\Hom(\pi^*E,\pi^*E))\,|\, \sigma(t\xi)=\sigma(\xi) \mbox{ for all
  $t>0$ and all $\xi\not=0$}\}$. 
Here $\pi:S^*M \to M$ is the cotangent bundle with the zero-section removed.
The symbol map $\sigma_\bullet$, $P \mapsto \sigma_P$, can be characterized as
follows:

\begin{lem}\label{lem:symbol}
Let $P\in\PPb(M,E)$, let $h\in C^\infty(M)$ be a real function and let $f\in
C^\infty(M,E)$ be a section such that
$dh(x)\not=0$ for all $x\in\supp(f)$.
Then
$$
e^{-ith}P(e^{ith}f) \xrightarrow{L^2} \sigma_P(dh)f
$$
as $t\to\infty$.
\end{lem}

\begin{proof}
The statement is known to hold if $P\in\PP(M,E)$, compare \cite[Sec.~2]{Ho}.
In this case the convergence is uniform.

Let $P\in\PPb(M,E)$ and choose $P_j\in \PP(M,E)$ such that $P_j \to P$ in the
$L^2$-norm topology. 
Then $\sigma_{P_j} \to \sigma_P$ uniformly.
Fix $\eps>0$.
Choose $j$ so large that 
$$
\|P-P_j\| \leq \eps
$$
and 
$$
\|\sigma_P-\sigma_{P_j}\|_{C^0} \leq \eps.
$$
Now choose $T$ so large that 
$$
\|\sigma_{P_j}(dh)f - e^{-ith}P_j(e^{ith}f) \|_{C^0} \leq \eps
$$
for all $t\geq T$.
Then we have for such $t$
\begin{eqnarray*}
\lefteqn{\|\sigma_{P}(dh)f - e^{-ith}P(e^{ith}f) \|_{L^2}}\\
&\leq&
\|\sigma_P(dh)f-\sigma_{P_j}(dh)f\|_{L^2}
+ \|\sigma_{P_j}(dh)f - e^{-ith}P_j(e^{ith}f) \|_{L^2}\\
&&
+ \|e^{-ith}(P_j-P)(e^{ith}f)\|_{L^2}\\
&\leq&
\eps \cdot \|dh\|_{C^0}\cdot \|f\|_{L^2}
+ \eps \cdot \sqrt{\vol(M)}
+ \eps \cdot \|f\|_{L^2}\\
&=&
C\cdot \eps
\end{eqnarray*}
where the constant $C$ depends only on $M$, $h$, and $f$.
This proves the lemma.
\end{proof}

Both $\PPb(M,E)$ and $\Symb(M,E)$ are $C^*$-algebras and the symbol map is a
homomorphism, $\sigma_{P\circ Q}(\xi) =
\sigma_{P}(\xi)\circ\sigma_{Q}(\xi)$.
In particular, if $\Phi\in C^0(M,\Hom(E',E))$ and $\Psi\in C^0(M,\Hom(E,E'))$,
then
$$
\sigma_{\Psi\circ P\circ \Phi}(\xi) = \Psi(x)\cdot\sigma_P(\xi)\cdot\Phi(x)
$$
holds for all $\xi\in S^*_xM$.
The following lemma says that this is still true if $\Phi$ and $\Psi$ are only
$L^\infty$ provided the left hand side makes sense.

\begin{lem}\label{lem:irreg}
Let $P\in\PP(M,E)$, let $\Phi\in L^\infty(M,\Hom(E',E))$, and let $\Psi\in
L^\infty(M,\Hom(E,E'))$.
Suppose $Q:=\Psi\circ P\circ \Phi \in \PPb(M,E')$.
Then 
$$
\sigma_Q(\xi) = \Psi(x)\cdot\sigma_P(\xi)\cdot\Phi(x)
$$
holds for almost all $x\in M$ and $\xi \in S^*_xM$.
\end{lem}

\begin{proof}
Since $M$ is compact $L^\infty\subset L^2$.
Choose $\Phi_j\in C^\infty(M,\Hom(E',E))$ and $\Psi_j\in
C^\infty(M,\Hom(E,E'))$ such that $\Phi_j \xrightarrow{L^2} \Phi$ and $\Psi_j
\xrightarrow{L^2} \Psi$ as $j\to\infty$. 
Note that $Q_j := \Psi_j\circ P\circ \Phi_j \in \PP(M,E')$ and 
\begin{equation}\label{eq:spj}
\sigma_{Q_j}(\xi) = \Psi_j(x)\cdot\sigma_P(\xi)\cdot\Phi_j(x).
\end{equation}
Put $r_j:=\Psi-\Psi_j$ and $q_j:=\Phi-\Phi_j$.
Let $h\in C^\infty(M)$ be a real function such that its critical set $\{x\in
M\,|\,dh(x)=0\}$ consists of finitely many points, e.~g.\ $h$ can be any Morse
function. 
Let $f,g\in C^\infty(M,E')$ vanish on a neighborhood of the critical set
$\{dh=0\}$. 
Then $x \mapsto \< \sigma_{Q_j}(dh(x))f(x),g(x)\>$ are well-defined smooth
functions.
From (\ref{eq:spj}) and $L^2$-convergence of $\Phi_j$ and $\Psi_j$ we conclude
$$
\<\sigma_{Q_j}(dh)f,g\> \xrightarrow{L^1} 
\<\Psi\sigma_P(dh)\Phi f,g\>
$$
as $j\to\infty$ and hence
\begin{equation}\label{eq:spjInt}
(\sigma_{Q_j}(dh)f,g)_{L^2} 
= \int_M \<\sigma_{Q_j}(dh(x))f(x),g(x)\>\, dx
\xrightarrow{j\to\infty} (\Psi\sigma_P(dh)\Phi f,g)_{L^2} .
\end{equation}

On the other hand, we compute using Lemma~\ref{lem:symbol}
\begin{eqnarray*}
\lefteqn{(\sigma_{Q}(dh)f,g)_{L^2}}\\
&=&
\int_M \< {\LLlim_{t\to\infty}} e^{-ith}Q(e^{ith}f),g\>\, dx\\
&=&
\int_M \<\LLlim_{t\to\infty} e^{-ith}
(\Psi_j+r_j)P((\Phi_j+q_j)e^{ith}f),g\>\, dx\\ 
&=&
(\sigma_{Q_j}(dh)f,g)_{L^2}
+ \int_M \<\LLlim_{t\to\infty} e^{-ith}
(\Psi_j P q_j + r_jP\Phi)e^{ith}f,g\>\, dx\\
&=&
(\sigma_{Q_j}(dh)f,g)_{L^2}
+ \int_M \Llim_{t\to\infty} e^{-ith}\<
(\Psi_j P q_j + r_jP\Phi)e^{ith}f,g\>\, dx\\
&=&
(\sigma_{Q_j}(dh)f,g)_{L^2}
+ \lim_{t\to\infty}\int_M  e^{-ith(x)}\<
(\Psi_j P q_j + r_jP\Phi)e^{ith}f|_{x},g(x)\>\, dx .
\end{eqnarray*}
Now fix $\eps>0$ and choose $j_0$ such that $\|r_j\|_{L^2}, \|q_j\|_{L^2} <
\eps$ for all $j\geq j_0$.
Then we have for all $j\geq j_0$
\begin{eqnarray}
\lefteqn{\left|\lim_{t\to\infty} \int_M e^{-ith(x)}
\<(\Psi_j P q_j + r_jP\Phi)e^{ith}f)|_{x},g(x)\>\, dx\right|}\nonumber\\
&\leq&
\limsup_{t\to\infty} \int_M \left|e^{-ith(x)}
\<(\Psi_j P q_j + r_jP\Phi)e^{ith}f)|_{x},g(x)\>\right|\, dx\nonumber\\
&=&
\limsup_{t\to\infty} \int_M \left|
\<(\Psi_j P q_j + r_jP\Phi)e^{ith}f)|_{x},g(x)\>\right|\, dx .
\label{eq:rest}
\end{eqnarray}
Denote the $L^2$-operator norm of $P$ by $C$.
Then we have for all $t>0$
\begin{eqnarray}
\int_M \left|\<(\Psi_j P q_j)e^{ith}f)|_{x},g(x)\>\right|\, dx 
&\leq&
\|P(q_j e^{ith} f)\|_{L^2} \cdot \|\Psi_j^*g\|_{L^2}\nonumber\\ 
&\leq&
C \cdot\|q_j\|_{L^2}\cdot \|f\|_{L^\infty}\cdot\|\Psi_j^*\|_{L^2}\cdot\|g\|_{L^\infty}\nonumber\\ 
&\leq&
C\cdot\eps\cdot\|f\|_{L^\infty}\cdot(\|\Psi\|_{L^2}+\eps)\cdot\|g\|_{L^\infty}
\label{eq:summand1}
\end{eqnarray}
and
\begin{eqnarray}
 \int_M \left|\<(r_jP\Phi)e^{ith}f)|_{x},g(x)\>\right|\, dx
&\leq&
\|P\Phi e^{iht} f\|_{L^2}\cdot\|r_j^*g\|_{L^2}\nonumber\\
&\leq&
C \cdot\|\Phi e^{ith} f\|_{L^2}\cdot\|r_j^*\|_{L^2}\cdot\|g\|_{L^\infty}\nonumber\\
&\leq&
C \cdot\|\Phi\|_{L^2}\cdot\|f\|_{L^\infty}\cdot\eps\cdot\|g\|_{L^\infty}.
\label{eq:summand2}
\end{eqnarray}
Plugging (\ref{eq:summand1}) and (\ref{eq:summand2}) into (\ref{eq:rest})
we find that 
$$
\left|\lim_{t\to\infty} \int_M e^{-ith(x)}
\<(\Psi_j P q_j + r_jP\Phi)e^{ith}f)|_{x},g(x)\>\, dx\right|
\quad\leq\quad
C'\cdot(\eps + \eps^2)
$$
for all $j\geq j_0$ with a constant $C'$ independent of $\eps$, $j_0$, and $j$.
Thus
$$
\lim_{j\to\infty}(\sigma_{Q_j}(dh)f,g)_{L^2} =
(\sigma_{Q}(dh)f,g)_{L^2}.
$$
By (\ref{eq:spjInt}) this means
$$
(\sigma_{Q}(dh)f,g)_{L^2} = 
(\Psi\sigma_P(dh)\Phi f,g)_{L^2}.
$$
Therefore
$$
\sigma_{Q}(dh(x)) = \Psi(x)\sigma_P(dh(x))\Phi(x)
$$
for almost all $x\in M$ with $dh(x)\not=0$.
Since $h$ is arbitrary
$$
\sigma_{Q}(\xi) = \Psi(x)\sigma_P(\xi)\Phi(x)
$$
for almost all $x\in M$ and $\xi \in S^*_xM$.
\end{proof}

\end{document}